\documentclass[letterpaper,oneside,english]{amsart}
\usepackage[T1]{fontenc}
\usepackage[latin9]{inputenc}
\usepackage{units}
\usepackage{mathrsfs}
\usepackage{amsthm}
\usepackage{amsbsy}
\usepackage{amstext}
\usepackage{amssymb}
\usepackage{esint}

\usepackage{graphicx}
\usepackage{lipsum}

\makeatletter

\numberwithin{equation}{section}
\numberwithin{figure}{section}
\theoremstyle{plain}
\newtheorem{thm}{\protect\theoremname}
\theoremstyle{definition}

\theoremstyle{remark}
\newtheorem{rem}[thm]{\protect\remarkname}
\theoremstyle{plain}
\newtheorem{lem}[thm]{\protect\lemmaname}
\ifx\proof\undefined
\newenvironment{proof}[1][\protect\proofname]{\par
\normalfont\topsep6\p@\@plus6\p@\relax
\trivlist
\itemindent\parindent
\item[\hskip\labelsep
\scshape
#1]\ignorespaces
}{%
\endtrivlist\@endpefalse
}
\providecommand{\proofname}{Proof}
\fi
\theoremstyle{plain}

\makeatother

\usepackage{babel}
\providecommand{\definitionname}{Definition}
\providecommand{\lemmaname}{Lemma}
\providecommand{\propositionname}{Proposition}
\providecommand{\remarkname}{Remark}
\providecommand{\theoremname}{Theorem}

\numberwithin{equation}{section} 
\numberwithin{thm}{section}

\begin{document}

\title[Existence of $f$ self-shrinkers with a conical end]{Existence of self-shrinkers to the degree-one curvature flow with a rotationally symmetric conical end}

\author{Siao-Hao Guo}
\begin{abstract}
Given a smooth, symmetric, homogeneous of degree one function $f\left(\lambda_{1},\cdots,\,\lambda_{n}\right)$
satisfying $\partial_{i}f>0$ for all $i=1,\cdots,\,n$, and a rotationally symmetric 
cone $\mathcal{C}$ in $\mathbb{R}^{n+1}$,
we show that there is a $f$
self-shrinker (i.e. a hypersurface $\Sigma$ in $\mathbb{R}^{n+1}$
which satisfies $f\left(\kappa_{1},\cdots,\,\kappa_{n}\right)+\frac{1}{2}X\cdot N=0$, where $X$ is the position vector, $N$ is the unit normal vector, and $\kappa_{1},\cdots,\,\kappa_{n}$ are principal curvatures
of $\Sigma$) that is asymptotic to $\mathcal{C}$
at infinity.  
\end{abstract}
\maketitle

\section{Introduction}

Let $\mathcal{C}$ be a rotationally symmetric cone in $\mathbb{R}^{n+1}$, say 
\[
\mathcal{C}=\left\{ \left(\sigma s\,\nu,\, s\right)\Big|\,\nu\in\mathcal{\mathbf{S}}^{n-1},\, s\in\mathbb{R}_{+}\right\} 
\]
where $\sigma>0$ is a constant. Let $\boldsymbol{\Sigma}$ be a properly embedded hypersurface in $\mathbb{R}^{n+1}$. Then $\boldsymbol{\Sigma}$ is called a self-shrinker to the mean curvature flow (MCF: motions of hypersurfaces whose normal velocity are given by the mean curvature vector) which is $C^{k}$ asymptotic to the cone $\mathcal{C}$ at infinity provided that
\[
H\,+\,\frac{1}{2}X\cdot N=0
\]
\[
\varrho\boldsymbol{\Sigma}\,\overset{C_{\textrm{loc}}^{k}}{\longrightarrow}\,\mathcal{C}\quad\textrm{as}\;\,\varrho\searrow0
\]
where $X$ is the position vector, $N$ is the unit normal vector and $H$ is the mean curvature of $\boldsymbol{\Sigma}$. Note that the rescaled family of hypersurfaces $\left\{ \boldsymbol{\Sigma}_{t}=\sqrt{-t}\,\boldsymbol{\Sigma}\right\}_{-1\leq t<0} $
forms a MCF starting from $\boldsymbol{\Sigma}$ (when $t=-1$) and converging locally $C^{k}$ to $\mathcal{C}$ as $t\nearrow0$.

In the case when $\boldsymbol{\Sigma}$ is rotationally symmetric, one can
parametrize it by 
\[
X\left(\nu,\, s\right)=\left(\mathtt{r}\left(s\right)\,\nu,\, s\right)\quad\textrm{for}\,\;\nu\in\mathcal{\mathbf{S}}^{n-1},\, s\in\left(c_{1},\, c_{2}\right)
\]
for some constants $0\leq c_{1}<c_{2}\leq\infty$. We may orient it
by the unit-normal 
\begin{equation}
N=\frac{\left(-\nu,\,\partial_{\text{s}}\mathtt{r}\right)}{\sqrt{1+\left(\partial_{\text{s}}\mathtt{r}\right){}^{2}}}\label{154}
\end{equation}
At each point $X\in\boldsymbol{\Sigma}$, choose an orthonormal basis $\left\{ e_{1},\cdots, e_{n}\right\} $
for $T_{X}\text{\ensuremath{\boldsymbol{\Sigma}}}$ so that 
\[
e_{n}=\,\frac{\partial_{s}X}{|\partial_{s}X|}\,=\frac{\left(\partial_{\text{s}}\mathtt{r}\,\nu,\,1\right)}{\sqrt{1+\left(\partial_{\text{s}}\mathtt{r}\right){}^{2}}}
\]
then $\left\{ e_{1},\cdots, e_{n}\right\} $ forms a set of principal
vectors of $\boldsymbol{\Sigma}$ at $X$ with principal curvatures
\begin{equation}
\kappa_{1}=\cdots=\kappa_{n-1}=\frac{1}{\mathtt{r}\sqrt{1+\left(\partial_{s}\mathtt{r}\right){}^{2}}},\quad\kappa_{n}=\frac{-\partial_{s}^{2}\mathtt{r}}{\left(1+\left(\partial_{s}\mathtt{r}\right){}^{2}\right)^{\frac{3}{2}}}\label{156}
\end{equation}
As a result, $\boldsymbol{\Sigma}$ is a rotationally symmetric self-shrinker to the MCF if and only if

\begin{equation}
\left(\frac{n-1}{\mathtt{r}}-\frac{\partial_{s}^{2}\mathtt{r}}{1\,+\,|\partial_{\text{s}}\mathtt{r}|{}^{2}}\right)+\,\frac{1}{2}\left(s\,\partial_{s}\mathtt{r}-\mathtt{r}\right)=0 \label{155}
\end{equation}
Kleene and Moller showed in \cite{KM} that there
exists a unique rotationally symmetric self-shrinker
\[
\boldsymbol{\Sigma}:\; X\left(\nu,\, s\right)=\left(\mathtt{r}(s)\nu,\, s\right),\quad\nu\in\mathcal{\mathbf{S}}^{n-1},\, s\in[R,\,\infty)
\]
where the radius function $\mathtt{r}(s)$ satisfies (\ref{155}) and
\[
s\Big|\mathtt{r}(s)-\sigma s\Big|\,\leq\frac{2\left(n-1\right)}{\sigma},\quad s^{2}\Big|\partial_{s}\mathtt{r}-\sigma\Big|\,\leq\frac{2\left(n-1\right)}{\sigma}
\]
The key step is to analyze the following representation formula for (\ref{155}):

\[
\mathtt{r}(s)=\,\sigma s \,+
\]
\noindent \resizebox{1.0\linewidth}{!}{
  \begin{minipage}{\linewidth}
  \begin{align*}
\, s\int_{s}^{\infty}\frac{1}{x^{2}}\left\{ \int_{x}^{\infty}\xi\,\exp\left(-\frac{1}{2}\int_{x}^{\xi}\boldsymbol{\tau}\left(1+\left(\partial_{s}\mathtt{r}\left(\boldsymbol{\tau}\right)\right)^{2}\right)d\boldsymbol{\tau}\right)\,\left[\frac{n-1}{\mathtt{r}\left(\xi\right)}\left(1+\left(\partial_{s}\mathtt{r}\left(\xi\right)\right)^{2}\right)\right]\, d\xi\right\} dx
\end{align*}
  \end{minipage}
}
\vspace{0.2in}

On the other hand, let $f\left(\lambda_{1},\cdots,\,\lambda_{n}\right)$ be a symmetric and homogeneous of degree-one function which satisfies 
\[
\partial_{i}f>0\quad\forall\:\, i=1,\cdots, n
\]
Note that by the properties of $f$, we may assume that its domain of definition $\boldsymbol{\Omega}$ is closed under permutation and homothety, i.e.
\[
\left(\lambda_{\sigma (1)},\cdots,\,\lambda_{\sigma (n)}\right),\, \left(\varrho\lambda_{1},\cdots,\,\varrho\lambda_{n}\right) \in \boldsymbol{\Omega} \:\quad \forall \:\, \sigma \in \textbf{S}_n,\,\varrho>0
\]
whenever $\left(\lambda_{1},\cdots,\,\lambda_{n}\right) \in \boldsymbol{\Omega}$, where $\textbf{S}_n$ is the symmetric group.
Andrews in \cite{A} considered the following
evolution of hypersurfaces in $\mathbb{R}^{n+1}$:
\[
\partial_{t}X_t^{\bot}=f\left(\kappa_{1},\cdots,\,\kappa_{n}\right)N
\]
where $\kappa_{1},\cdots,\,\kappa_{n}$ are the principal curvatures
of the evolving hypersurface. In particular, if we take the curvature
function to be $f\left(\lambda_{1},\cdots,\,\lambda_{n}\right)=\lambda_{1}+\cdots+\lambda_{n}$,
then it reduces to the MCF. We call a
hypersurface $\Sigma$ in $\mathbb{R}^{n+1}$ to be a ``$f$
self-shrinker'' provided that
\[
f\left(\kappa_{1},\cdots,\,\kappa_{n}\right)\,+\,\frac{1}{2}X\cdot N=0
\]
Just like the MCF, the rescaled family of a $f$ self-shrinkers
is a self-similar solution to the ``$f$ curvature flow''. More precisely, the family of hypersurfaces $\left\{ \Sigma_{t}=\sqrt{-t}\,\Sigma\right\} _{t<0}$
defines a $f$ curvature flow. Moreover, if $\Sigma$ is $C^{k}$ asymptotic
to the cone $\mathcal{C}$ at infinity, the flow $\left\{ \Sigma_{t}\right\} _{t<0}$
converges locally $C^{k}$ to $\mathcal{C}$ as $t\nearrow0$.

This paper is an extension of the existence result in \cite{KM} to the class of $f$ self-shrinkers
with a tangent cone at infinity. In fact, the motivation of this paper comes from \cite{G}, in which we generalized the uniqueness 
result of self-shrinkers with a conical end for the MCF in \cite{W} to the $f$ curvature flow. 
Below is our main result.

\begin{thm}\label{t7}
Assume that $f$ is $C^{\mathsf{k}+1}$ in a bounded neighborhood
$\mathcal{K}$ of $\left(\overrightarrow{1},\,0\right)\equiv \left(1,\cdots,\,1,\,0\right)\in\mathbb{R}^{n}$
with $\mathsf{k}\geq3$. Then there exist $R=R\left(n,\,\mathsf{k},\mathcal{\, C},\,\mathcal{K},\,\parallel f\parallel_{C^{\mathsf{k}+1}\left(\mathcal{\mathit{\mathcal{K}}}\right)}\right)\geq1$
and $\overset{\circ}{\mathsf{u}}\in C_{0}^{\mathsf{k}}[R,\,\infty)$ so that 
\[
\Sigma\equiv\left\{ \left.  \left(\left(\sigma s\,+\,\frac{f\left(\overrightarrow{1},\,0\right)}{\sigma s}\,+\,\overset{\circ}{\mathsf{u}}\left(s\right)\right)\nu,\; s\right)\right|\,\nu\in\mathcal{\mathbf{S}}^{n-1},\, s\in[R,\,\infty)\right\} 
\]
is a rotationally symmetric $f$ self-shrinker which is $C^{\mathsf{k}}$ asymptotic
to $\mathcal{C}$ at infinity. \\
In addition, the corresponding self-similar solution to the $f$ curvature flow is given by 

\noindent \resizebox{1.01\linewidth}{!}{
  \begin{minipage}{\linewidth}
  \begin{align*}
\Sigma_{t}=\sqrt{-t}\,\Sigma=\left\{ \left. \left(\left(\sigma s\,-t\,\frac{f\left(\overrightarrow{1},\,0\right)}{\sigma s}\,+\,\overset{\circ}{\mathsf{u}}_{t}\left(s\right)\right)\nu,\; s\right)\right|\,\nu\in\mathcal{\mathbf{S}}^{n-1},\, s\in[\sqrt{-t}R,\,\infty)\right\} 
\end{align*}
  \end{minipage}
}\\
for $t\in[-1,\,0)$, where $\overset{\circ}{\mathsf{u}}_{t}\left(s\right)\equiv \sqrt{-t}\,\overset{\circ}{\mathsf{u}}\left(\frac{s}{\sqrt{-t}}\right)$ satisfies
\[
\parallel s^{3}\,\overset{\circ}{\mathsf{u}}_{t}\parallel_{L^{\infty}[\sqrt{-t}R,\,\infty)}+\parallel s^{4}\partial_{s}\overset{\circ}{\mathsf{u}}_{t}\parallel_{L^{\infty}[\sqrt{-t}R,\,\infty)}+\cdots+\parallel s^{k+2}\partial_{s}^{k-1}\overset{\circ}{\mathsf{u}}_{t}\parallel_{L^{\infty}[\sqrt{-t}R,\,\infty)}
\]
\[
\leq\, C\left(n,\,\mathsf{k},\mathcal{\, C},\,\parallel f\parallel_{C^{\mathsf{k}}\left(\mathcal{K}\right)}\right)\,\left(-t\right)^{2}
\]
\[
\parallel s^{k+1}\partial_{s}^{k}\overset{\circ}{\mathsf{u}}_{t}\parallel_{L^{\infty}[\sqrt{-t}R,\,\infty)}\,\leq\, C\left(n,\,\mathsf{k},\mathcal{\, C},\,\parallel f\parallel_{C^{\mathsf{k}}\left(\mathcal{K}\right)}\right)\,\left(-t\right)
\]
for all $t\in[-1,\,0)$.
\end{thm} 
Note that in view of the principal curvatures of $\mathcal{C}$:  
\[
\kappa_{1}^{\mathcal{C}}=\cdots=\kappa_{n-1}^{\mathcal{C}}=\frac{1}{\sigma s\sqrt{1+\sigma{}^{2}}},\quad\kappa_{n}^{\mathcal{C}}=0
\]
(see (\ref{156})) and the homogeneity of the $f$, the condition that $f$ is $C^{\mathsf{k}+1}$ in a bounded neighborhood
$\mathcal{K}$ of $\left(\overrightarrow{1},\,0\right)=\left(1,\cdots,\,1,\,0\right)\in\mathbb{R}^{n}$ yields that $f$ is $C^{\mathsf{k}+1}$ 
in an open set containing all the the principal curvature vectors $\left(\kappa_{1}^{\mathcal{C}},\cdots,\,\kappa_{n}^{\mathcal{C}}\right)$
of $\mathcal{C}$. 

In the next section, we would use a similar representation formula as in  \cite{KM} (see (\ref{180})) to study the ODE corresponding to the $f$ self-shrinker equation (see (\ref{157})). 
We then use that, together with Banach fixed point theorem, to prove the existence of $f$ self-shrinkers.

\section*{Acknowledgement}
I am grateful to my advisor, Natasa Sesum, for her great support and many helpful suggestions to this paper.

\vspace{.5in} 
\section{Proof of the main theorem}
For a hypersurface of revolution $\Sigma$ in $\mathbb{R}^{n+1}$, we may parametrize it by 
\[
X\left(\nu,\, s\right)=\left(r\left(s\right)\,\nu,\, s\right)\quad\textrm{for}\;\nu\in\mathcal{\mathbf{S}}^{n-1},\, s\in\left(c_{1},\, c_{2}\right)
\]
for some constants $0\leq c_{1}<c_{2}\leq\infty$. By (\ref{154}) and (\ref{156}),
$\Sigma$ is a rotationally symmetric $f$ self-shrinker if and only if
\begin{equation}
f\left(\frac{1}{r\sqrt{1+\left(\partial_{s}r\right){}^{2}}}\overrightarrow{1},\;\frac{-\partial_{s}^{2}r}{\left(1+\left(\partial_{s}r\right){}^{2}\right)^{\frac{3}{2}}}\right)+\;\frac{1}{2}\,\frac{s\,\partial_{s}r-r}{\sqrt{1+\left(\partial_{s}r\right){}^{2}}}=0\label{157}
\end{equation}
where $\overrightarrow{1}=\left(1,\cdots,1\right)\in\mathbb{R}^{n-1}$. By the homogeneity of $f$, (\ref{157}) is equivalent to
\begin{equation}
\mathcal{G}\left(\partial_{s}^{2}r,\,\partial_{\text{s}}r,\, r;\, s\right)=0\label{158}
\end{equation}
where 
\begin{equation}
\mathcal{G}\left(q,\, p,\, z;\, s\right)=\, f\left(\frac{1}{z}\overrightarrow{1},\;\frac{-q}{1+p^{2}}\right)+\;\frac{1}{2}\left(sp-z\right)\label{159}
\end{equation}

On the other hand, $\Sigma$ is $C^{\mathsf{k}}$ asymptotic to $\mathcal{C}$ at infinity if and only if
\begin{equation}
\varrho\, r\left(\frac{s}{\varrho}\right)-\sigma s\,\overset{C_{\textrm{loc}}^{\mathsf{k}}}{\longrightarrow}\,0\quad\textrm{as}\;\varrho\searrow0\label{160}
\end{equation}
Let 
\begin{equation}
u\left(s\right)=r\left(s\right)-\sigma s\label{161}
\end{equation}
be the difference between the solution $r\left(s\right)$ and the radius function of the cone. Then (\ref{160}) is translated into
\begin{equation}
\varrho\, u\left(\frac{s}{\varrho}\right)\,\overset{C_{\textrm{loc}}^{\mathsf{k}}}{\longrightarrow}\,0\quad\textrm{as}\,\;\varrho\searrow0\label{162}
\end{equation}
which is equivalent to 
\[
u\left(s\right)=o\left(s\right),\;\partial_{s}u=o\left(1\right),\cdots,\;\partial_{s}^{\mathsf{k}}u=o\left(s^{1-\mathsf{k}}\right)\quad\textrm{as}\; s\nearrow\infty
\]

Now we would like to derive an equation for $u$ by plugging $r\left(s\right)=\sigma s+u\left(s\right)$
into (\ref{158}) and using Taylor's theorem to expand.
For ease of notation, let
\begin{equation}
\textbf{z}\left(u\left(s\right);\,\theta\right)=\sigma s+\theta\, u\left(s\right)\label{163}
\end{equation}
\begin{equation}
\textbf{p}\left(u\left(s\right);\,\theta\right)=\partial_{\text{s}}\left(\sigma s+\theta\, u\left(s\right)\right)=\sigma+\theta\,\partial_{s}u\label{164}
\end{equation}
\begin{equation}
\textbf{q}\left(u\left(s\right);\,\theta\right)=\partial_{s}^{2}\left(\sigma s+\theta\, u\left(s\right)\right)=\theta\,\partial_{s}^{2}u\label{165}
\end{equation}
Note that $\theta=0$ and $\theta=1$ correspond to the radius function of the cone and the solution $r\left(s\right)$, respectively. 
Then (\ref{158}) can be written as 
\[
\mathcal{G}\left(\textbf{q}\left(u\left(s\right);\,1\right),\;\textbf{p}\left(u\left(s\right);\,1\right),\;\textbf{z}\left(u\left(s\right);\,1\right);\, s\right)=0
\]
By Taylor's theorem, (\ref{159}), (\ref{163}), (\ref{164}), (\ref{165})
and the homogeneity of $f$ (noting that $f$ has
degree $1$, $\partial_{i}f$ has degree $0$ and $\partial_{ij}^{2}f$
has degree $-1$), there holds 
\begin{equation}
0=\mathcal{G}\left(\textbf{q}\left(u;\,1\right),\;\textbf{p}\left(u;\,1\right),\;\textbf{z}\left(u;\,1\right);\, s\right)\label{166}
\end{equation}
\[
=\mathcal{G}\left(\textbf{q}\left(u;\,0\right),\;\textbf{p}\left(u;\,0\right),\;\textbf{z}\left(u;\,0\right);\, s\right)+\,\partial_{\theta}\left\{ \mathcal{G}\left(\textbf{q}\left(u;\,\theta\right),\;\textbf{p}\left(u;\,\theta\right),\;\textbf{z}\left(u,\,\theta\right);\, s\right)\right\} \Big|_{\theta=0}
\]
\[
+\,\int_{0}^{1}\,\partial_{\theta}^{2}\left\{ \mathcal{G}\left(\textbf{q}\left(u;\,\theta\right),\;\textbf{p}\left(u;\,\theta\right),\;\textbf{z}\left(u;\,\theta\right);\, s\right)\right\} \,\left(1-\theta\right)d\theta
\]
\[
=f\left(\frac{1}{\sigma s}\overrightarrow{1},\,0\right)-\,\frac{\partial_{n}f\left(\frac{1}{\sigma s}\overrightarrow{1},\,0\right)}{1+\sigma^{2}}\partial_{s}^{2}u-\,\sum_{i=1}^{n-1}\frac{\partial_{i}f\left(\frac{1}{\sigma s}\overrightarrow{1},\,0\right)}{\sigma^{2}s^{2}}u+\,\frac{1}{2}\left(s\,\partial_{s}u-u\right)+\,\mathcal{Q}u
\]
\[
=\frac{f\left(\overrightarrow{1},\,0\right)}{\sigma s}\, -\,\frac{\partial_{n}f\left(\overrightarrow{1},\,0\right)}{1+\sigma^{2}}\partial_{s}^{2}u-\,\sum_{i=1}^{n-1}\frac{\partial_{i}f\left(\overrightarrow{1},\,0\right)}{\sigma^{2}s^{2}}u+\,\frac{1}{2}\left(s\,\partial_{s}u-u\right)+\,\mathcal{Q}u
\]
where 
\[
\mathcal{Q}u=\int_{0}^{1}\,\partial_{\theta}^{2}\left\{ \mathcal{G}\left(\textbf{q}\left(u;\,\theta\right),\;\textbf{p}\left(u;\,\theta\right),\;\textbf{z}\left(u;\,\theta\right);\, s\right)\right\} \,\left(1-\theta\right)d\theta
\]
Let
\begin{equation}
\omega = \omega\left(u\left(s\right);\,\theta\right)\equiv\sigma s\left(\frac{1}{\textbf{z}\left(u\left(s\right);\,\theta\right)}\overrightarrow{1},\;\frac{-\textbf{q}\left(u\left(s\right);\,\theta\right)}{1+\textbf{p}\left(u\left(s\right);\,\theta\right)^{2}}\right)\label{168}
\end{equation}
\[
=\left(\frac{1}{1+\theta\frac{u}{\sigma s}}\overrightarrow{1},\;\frac{-2\theta\sigma s\,\partial_{s}^{2}u}{1+\left(\sigma+\theta\partial_{s}u\right)^{2}}\right)
\]
By some simple calculations and the homogeneity of $f$, one get
\begin{equation}
\mathcal{Q}u(s)\,=\,\left\{\left(\int_{0}^{1}\frac{\sigma s\,\partial_{nn}^{2}f\circ\omega}{\left(1+\textbf{p}^{2}\right)^{2}}\,\left(1-\theta\right)d\theta\right)\left(\partial_{s}^{2}u\right)^{2}\right.\label{167}
\end{equation}
\[
+\left(\int_{0}^{1}\left(\sigma s\,\partial_{nn}^{2}f\circ\omega\frac{4\textbf{q}^{2}\textbf{p}^{2}}{\left(1+\textbf{p}^{2}\right)^{4}}+\,\partial_{n}f\circ\omega\frac{2\textbf{q}\left(1-3\textbf{p}^{2}\right)}{\left(1+\textbf{p}^{2}\right)^{3}}\right)\left(1-\theta\right)d\theta\right)\left(\partial_{s}u\right)^{2}
\]
\[
+\left(\int_{0}^{1}\frac{\left(\sum_{i,\, j=1}^{n-1}\sigma s\,\partial_{ij}^{2}f\circ\omega\right)+\,2\textbf{z}\sum_{i=1}^{n-1}\partial_{i}f\circ\omega}{\textbf{z}^{4}}\left(1-\theta\right)d\theta\right)u^{2}
\]
\[
+2\left(\int_{0}^{1}\left(\sigma s\,\partial_{nn}^{2}f\circ\omega\frac{-2\textbf{q}\textbf{p}}{\left(1+\textbf{p}^{2}\right)^{3}}+\,\partial_{n}f\circ\omega\frac{2\textbf{p}}{\left(1+\textbf{p}^{2}\right)^{2}}\right)\left(1-\theta\right)d\theta\right)\partial_{s}^{2}u\,\partial_{s}u
\]

\noindent \resizebox{1.0\linewidth}{!}{
  \begin{minipage}{\linewidth}
  \begin{align*}
\left.+2\left(\int_{0}^{1}\,\sum_{i=1}^{n-1}\frac{\sigma s\,\partial_{ni}^{2}f\circ\omega}{\left(1+\textbf{p}^{2}\right)\textbf{z}^{2}}\,\left(1-\theta\right)d\theta\right)u\,\partial_{s}^{2}u-\,2\left(\int_{0}^{1}\,\sum_{i=1}^{n-1}\sigma s\,\partial_{ni}^{2}f\circ\omega\frac{2\textbf{q}\textbf{p}}{\left(1+\textbf{p}^{2}\right)^{2}\textbf{z}^{2}}\,\left(1-\theta\right)d\theta\right)u\,\partial_{s}u\right\}
\end{align*}
  \end{minipage}
}\\\\
where $\textbf{z}=\textbf{z}\left(u\left(s\right);\,\theta\right)$,
$\textbf{p}=\textbf{p}\left(u\left(s\right);\,\theta\right)$,
$\textbf{q}=\textbf{q}\left(u\left(s\right);\,\theta\right)$.

Now reorganize (\ref{166}) to get an equation of $u$ as follows:
\begin{equation}
\mathcal{L}u\,=\,\frac{1+\sigma^{2}}{\partial_{n}f\left(\overrightarrow{1},\,0\right)}\left(\frac{f\left(\overrightarrow{1},\,0\right)}{\sigma s}-\,\sum_{i=1}^{n-1}\frac{\partial_{i}f\left(\overrightarrow{1},\,0\right)}{\sigma^{2}s^{2}}u+\,\mathcal{Q}u\right)\label{169}
\end{equation}
where $\mathcal{L}$ is a linear differential operator defined by
\begin{equation}
\mathcal{L}u\,=\,\partial_{s}^{2}u-\,\frac{1}{2}\frac{1+\sigma^{2}}{\partial_{n}f\left(\overrightarrow{1},\,0\right)}\left(s\,\partial_{s}u-u\right)\label{170}
\end{equation}

To summarize, in order to find a rotationally symmetric $f$ self-shrinker
$\Sigma$ which is $C^{\mathsf{k}}$ asymptotic to $\mathcal{C}$
at infinity, it suffices to solve the ODE (\ref{158}) satisfying the condition
(\ref{160}). By (\ref{161}), that amounts to solving the problem
(\ref{169}) satisfying the condition (\ref{162}). We do this by regarding (\ref{169})
as a fixed point problem of a nonlinear map in a Banach space
where (\ref{162}) is satisfied and the nonlinear map
is a contraction. The existence of solutions is then
assured by Banach fixed point theorem.\\ 
To this end, we would first study the corresponding linear problem and get estimates of solutions, which are crucial to solving the nonlinear problem. Below are two lemmas to achieve that.

In the first lemma, we analyze the linear differential operator $\mathcal{L}$
in (\ref{170}). We derive a representation formula for the associated Cauchy
problem as in \cite{KM} (see (\ref{180})). Then we use that to estimate solutions.
\begin{lem}\label{l29}
Fix $R>0$, then for any function $\boldsymbol{\eta}\in C_{0}[R,\,\infty)$
(i.e. $\boldsymbol{\eta}\in C\,[R,\,\infty)$ and $\boldsymbol{\eta}\rightarrow0$
as $s\nearrow\infty$), there is a unique $C^{2}[R,\,\infty)$ solution
$\textrm{w}$ to the following problem:
\begin{equation}
\mathcal{L}\textrm{w}=\boldsymbol{\eta}\quad\textrm{on}\;[R,\,\infty)\label{171}
\end{equation}
\begin{equation}
\frac{\textrm{w}}{s}\rightarrow0, \,\left(s\,\partial_{s}\textrm{w}-\textrm{w}\right)\rightarrow0\quad\textrm{as}\; s\nearrow\infty\label{172}
\end{equation}
where $\mathcal{L}$ is defined in (\ref{170}).
Moreover, $\textrm{w}$
satisfies the following estimates. For any $\gamma\geq0$, there hold 
\begin{equation}
\max\left\{ \parallel s^{\gamma}\textrm{w}\parallel_{L^{\infty}[R,\,\infty)},\;\parallel s^{\gamma+1}\partial_{s}\textrm{w}\parallel_{L^{\infty}[R,\,\infty)}\right\} \,\leq\,4\frac{\partial_{n}f\left(\overrightarrow{1},\,0\right)}{1+\sigma^{2}}\parallel s^{\gamma}\boldsymbol{\eta}\parallel_{L^{\infty}[R,\,\infty)}\label{173}
\end{equation}
\begin{equation}
\parallel s^{\gamma}\partial_{s}^{2}\textrm{w}\parallel_{L^{\infty}[R,\,\infty)}\,\leq\,4\parallel s^{\gamma}\boldsymbol{\eta}\parallel_{L^{\infty}[R,\,\infty)}\label{174}
\end{equation}\\
In addition, $\textrm{w}\in C_{0}^{m+2}[R,\,\infty)$ whenever $\boldsymbol{\eta}\in C_{0}^{m}[R,\,\infty)$
for some $m\in\mathbb{N}$.\end{lem}
\begin{proof}
Firstly, suppose that $\textrm{w}\in C^{2}[R,\,\infty)$ is a solution to
the linear problem (\ref{171}) and (\ref{172}), it must satisfy the following. Given a sequence $\left\{ R_{j}\in\left(R,\,\infty\right)\right\} _{j\in\mathbb{N}}$
so that $R_{j}\nearrow\infty$ as $j\nearrow\infty$, then from
(\ref{170}), (\ref{171}) and (\ref{172}), we get
\begin{equation}
\partial_{s}^{2}\left(\frac{\textrm{w}}{s}\right)+\,\left(\frac{2}{s}-\,\frac{s}{2}\frac{1+\sigma^{2}}{\partial_{n}f\left(\overrightarrow{1},\,0\right)}\right)\partial_{s}\left(\frac{\textrm{w}}{s}\right)\,=\frac{\boldsymbol{\eta}}{s}\quad\textrm{for}\; s\in[R,\, R_{j})\label{175}
\end{equation}
\begin{equation}
\frac{\textrm{w}(R_{j})}{R_{j}}\rightarrow0,\, \left(R_{j}\,\partial_{s}\textrm{w}(R_{j})-\textrm{w}(R_{j})\right)\rightarrow0\quad\textrm{as}\; j\nearrow\infty\label{176}
\end{equation}
From (\ref{175}), which is a first order linear ODE in $\partial_{s}\left(\frac{\textrm{w}}{s}\right)$, we get

\noindent \resizebox{1.0\linewidth}{!}{
  \begin{minipage}{\linewidth}
  \begin{align}
\textrm{w}(s)=\, s\left\{ \frac{\textrm{w}(R_{j})}{R_{j}}-\,\left(R_{j}\,\partial_{s}\textrm{w}(R_{j})-\textrm{w}(R_{j})\right)\int_{s}^{R_{j}}x^{-2}\,\exp\left(-\frac{1}{2}\int_{x}^{R_{j}}\boldsymbol{\tau}\frac{1+\sigma^{2}}{\partial_{n}f\left(\overrightarrow{1},\,0\right)}d\boldsymbol{\tau}\right)dx\right\}\,+ \label{177}
\end{align}
  \end{minipage}
}

\[
s\int_{s}^{R_{j}}\frac{1}{x^{2}}\left(\int_{x}^{R_{j}}\xi\,\exp\left(-\frac{1}{2}\int_{x}^{\xi}\boldsymbol{\tau}\frac{1+\sigma^{2}}{\partial_{n}f\left(\overrightarrow{1},\,0\right)}d\boldsymbol{\tau}\right)\,\boldsymbol{\eta}(\xi)\, d\xi\right)dx 
\]

\noindent \resizebox{1.0\linewidth}{!}{
  \begin{minipage}{\linewidth}
  \begin{align*}
=s\left\{ \frac{\textrm{w}(R_{j})}{R_{j}}-\,\left(R_{j}\,\partial_{s}\textrm{w}(R_{j})-\textrm{w}(R_{j})\right)\int_{s}^{R_{j}}x^{-2}\,\exp\left(-\frac{1}{4}\frac{1+\sigma^{2}}{\partial_{n}f\left(\overrightarrow{1},\,0\right)}\left(R_{j}^{2}-x^{2}\right)\right)\, dx\right\}\,+
\end{align*}
  \end{minipage}
}

\[
s\int_{s}^{R_{j}}\frac{1}{x^{2}}\left(\int_{x}^{R_{j}}\xi\,\exp\left(-\frac{1}{4}\frac{1+\sigma^{2}}{\partial_{n}f\left(\overrightarrow{1},\,0\right)}\left(\xi^{2}-x^{2}\right)\right)\boldsymbol{\eta}(\xi)\, d\xi\right)dx
\]
Note that in the last two terms of (\ref{177}), there hold
\begin{equation}
\int_{s}^{R_{j}}x^{-2}\,\exp\left(-\frac{1}{4}\frac{1+\sigma^{2}}{\partial_{n}f\left(\overrightarrow{1},\,0\right)}\left(R_{j}^{2}-x^{2}\right)\right)\, dx\,\leq\,\int_{s}^{\infty}x^{-2}dx\,\,=\,\frac{1}{s}\label{178}
\end{equation}
and
\begin{equation}
\int_{x}^{R_{j}}\xi\,\exp\left(-\frac{1}{4}\frac{1+\sigma^{2}}{\partial_{n}f\left(\overrightarrow{1},\,0\right)}\left(\xi^{2}-x^{2}\right)\right)\, d\xi\label{179}
\end{equation}
\[
\leq\,\int_{x}^{\infty}\xi\,\exp\left(-\frac{1}{4}\frac{1+\sigma^{2}}{\partial_{n}f\left(\overrightarrow{1},\,0\right)}\left(\xi^{2}-x^{2}\right)\right)\, d\xi\,=\,\,2\frac{\partial_{n}f\left(\overrightarrow{1},\,0\right)}{1+\sigma^{2}}
\]
Fix $s\in[R,\,\infty)$, by (\ref{176}) and (\ref{178}), one could take a limit (as $j\nearrow\infty$) in
(\ref{177}) to get 
\begin{equation}
\textrm{w}(s)=\, s\int_{s}^{\infty}\frac{1}{x^{2}}\left(\int_{x}^{\infty}\xi\,\exp\left(-\frac{1}{4}\frac{1+\sigma^{2}}{\partial_{n}f\left(\overrightarrow{1},\,0\right)}\left(\xi^{2}-x^{2}\right)\right)\boldsymbol{\eta}(\xi)\, d\xi\right)dx \label{180}
\end{equation}

Conversely, if we define a function $\textrm{w}$ by (\ref{180}),
then $\textrm{w}\in C^{2}[R,\,\infty)$ and it satisfies
\begin{equation}
\partial_{s}\textrm{w}=\,\int_{s}^{\infty}\frac{1}{x^{2}}\left(\int_{x}^{\infty}\xi\,\exp\left(-\frac{1}{4}\frac{1+\sigma^{2}}{\partial_{n}f\left(\overrightarrow{1},\,0\right)}\left(\xi^{2}-x^{2}\right)\right)\boldsymbol{\eta}(\xi)\, d\xi\right)dx\,-\label{181}
\end{equation}
\[
\frac{1}{s}\int_{s}^{\infty}\xi\,\exp\left(-\frac{1}{4}\frac{1+\sigma^{2}}{\partial_{n}f\left(\overrightarrow{1},\,0\right)}\left(\xi^{2}-s^{2}\right)\right)\boldsymbol{\eta}(\xi)\, d\xi
\]
\begin{equation}
\partial_{s}^{2}\textrm{w}=\,\boldsymbol{\eta}(s)\,-\frac{1}{2}\frac{1+\sigma^{2}}{\partial_{n}f\left(\overrightarrow{1},\,0\right)}\int_{s}^{\infty}\xi\,\exp\left(-\frac{1}{4}\frac{1+\sigma^{2}}{\partial_{n}f\left(\overrightarrow{1},\,0\right)}\left(\xi^{2}-s^{2}\right)\right)\boldsymbol{\eta}(\xi)\, d\xi \label{182}
\end{equation}
From (\ref{170}), (\ref{180}), (\ref{181}) and (\ref{182}), we
immediately get (\ref{171}). To verify (\ref{172}), we use (\ref{179})
and that $\boldsymbol{\eta}$ vanishes at infinity: 
\[
\frac{\textrm{w}}{s}=\,\int_{s}^{\infty}\frac{1}{x^{2}}\left(\int_{x}^{\infty}\xi\,\exp\left(-\frac{1}{4}\frac{1+\sigma^{2}}{\partial_{n}f\left(\overrightarrow{1},\,0\right)}\left(\xi^{2}-x^{2}\right)\right)\boldsymbol{\eta}(\xi)\, d\xi\right)dx
\]
\[
\leq\,2\frac{\partial_{n}f\left(\overrightarrow{1},\,0\right)}{1+\sigma^{2}}\left(\sup_{\xi>s}|\boldsymbol{\eta}(\xi)|\right)\int_{s}^{\infty}\frac{dx}{x^{2}}\,\rightarrow \,0\,\quad\textrm{as}\,\; s\nearrow\infty
\]
\[
|s\,\partial_{s}\textrm{w}-\textrm{w}|\,\leq\,\int_{s}^{\infty}\xi\,\exp\left(-\frac{1}{4}\frac{1+\sigma^{2}}{\partial_{n}f\left(\overrightarrow{1},\,0\right)}\left(\xi^{2}-s^{2}\right)\right)|\boldsymbol{\eta}(\xi)|\, d\xi
\]
\[
\leq\,2\frac{\partial_{n}f\left(\overrightarrow{1},\,0\right)}{1+\sigma^{2}}\left(\sup_{\xi>s}|\boldsymbol{\eta}(\xi)|\right)\rightarrow \,0\,\quad\textrm{as}\,\; s\nearrow\infty
\]
Thus, (\ref{180}) is the unique solution to the linear problem (\ref{171}) and (\ref{172}). \\

Now given $\gamma\geq0$, we would like to verify (\ref{173}). Note
that we may assume $\parallel s^{\gamma}\boldsymbol{\eta}\parallel_{L^{\infty}[R,\,\infty)}<\infty$;
otherwise there is nothing to check. For each $s\in[R,\,\infty)$,
by (\ref{180}), (\ref{181}) and (\ref{179}), we have 
\begin{equation}
s^{\gamma}|\textrm{w}(s)|\,\leq\, s\int_{s}^{\infty}\frac{1}{x^{2}}\left(\int_{x}^{\infty}\xi\,\exp\left(-\frac{1}{4}\frac{1+\sigma^{2}}{\partial_{n}f\left(\overrightarrow{1},\,0\right)}\left(\xi^{2}-s^{2}\right)\right)\xi^{\gamma}|\boldsymbol{\eta}(\xi)|\, d\xi\right)dx\label{183}
\end{equation}
\[
\leq\, s\int_{s}^{\infty}\frac{2}{x^{2}}\,\frac{\partial_{n}f\left(\overrightarrow{1},\,0\right)}{1+\sigma^{2}}\,\left(\sup_{\xi\geq s}\left(\xi^{\gamma}|\boldsymbol{\eta}(\xi)|\right)\right)\, dx
\]
\[
=2\frac{\partial_{n}f\left(\overrightarrow{1},\,0\right)}{1+\sigma^{2}}\,\sup_{\xi\geq R}\left(\xi^{\gamma}|\boldsymbol{\eta}(\xi)|\right)
\]\\
and
\[
s^{\gamma+1}|\partial_{s}\textrm{w}(s)|\,\leq\, \left\{s\int_{s}^{\infty}\frac{1}{x^{2}}\left(\int_{x}^{\infty}\xi\,\exp\left(-\frac{1}{4}\frac{1+\sigma^{2}}{\partial_{n}f\left(\overrightarrow{1},\,0\right)}\left(\xi^{2}-s^{2}\right)\right)\xi^{\gamma}|\boldsymbol{\eta}(\xi)|\, d\xi\right)dx\right.
\]
\begin{equation}
\left.+\int_{s}^{\infty}\xi\,\exp\left(-\frac{1}{4}\frac{1+\sigma^{2}}{\partial_{n}f\left(\overrightarrow{1},\,0\right)}\left(\xi^{2}-s^{2}\right)\right)\xi^{\gamma}|\boldsymbol{\eta}(\xi)|\, d\xi \right\} \label{184}
\end{equation}
\[
\leq\,4\frac{\partial_{n}f\left(\overrightarrow{1},\,0\right)}{1+\sigma^{2}}\,\sup_{\xi\geq R}\left(\xi^{\gamma}|\boldsymbol{\eta}(\xi)|\right)
\]
As for (\ref{174}), we use the equation (\ref{171}) and (\ref{183}), (\ref{184}) to get 
\[
s^{\gamma}\left|\partial_{s}^{2}\textrm{w}(s)\right|\,=\, s^{\gamma}\left|\frac{1}{2}\frac{1+\sigma^{2}}{\partial_{n}f\left(\overrightarrow{1},\,0\right)}\left(s\,\partial_{s}\textrm{w}\left(s\right)-\textrm{w}\left(s\right)\right)+\boldsymbol{\eta}(s)\right|
\]
\[
\leq\,\frac{1}{2}\frac{1+\sigma^{2}}{\partial_{n}f\left(\overrightarrow{1},\,0\right)}\left(s^{\gamma+1}|\partial_{s}\textrm{w}(s)|\,+\, s^{\gamma}|\textrm{w}(s)|\right)\,\,+\,\, s^{\gamma}|\boldsymbol{\eta}(s)|
\]
\[
\leq\,\, 4\,\sup_{\xi\geq R}\left(\xi^{\gamma}|\boldsymbol{\eta}(\xi)|\right)
\]

Lastly, one could see from (\ref{182}) that $\textrm{w}\in C_{0}^{m+2}[R,\,\infty)$
as long as $\boldsymbol{\eta}\in C_{0}^{m}[R,\,\infty)$ for $m\in\mathbb{N}$.
\end{proof}
Next, let's consider the following normed space which we are going to work with.
Fix $R>0$, define
\[
\boldsymbol{\parallel}v\boldsymbol{\parallel}_{\Im}\,=\,\max\Big\{\,\parallel s\, v\parallel_{L^{\infty}[R,\,\infty)},\;\parallel s^{2}\partial_{s}v\parallel_{L^{\infty}[R,\,\infty)},\;\cdots,\;\parallel s^{\mathsf{k}}\partial_{s}^{\mathsf{k}-1}v\parallel_{L^{\infty}[R,\,\infty)},
\]
\begin{equation}
\parallel s^{\mathsf{k}}\partial_{s}^{\mathsf{k}}v\parallel_{L^{\infty}[R,\,\infty)}\,\Big\}\label{185}
\end{equation}
and a vector space 
\begin{equation}
\Im=\left\{ v\in C_{0}^{\mathsf{k}}[R,\,\infty)\Big|\,\boldsymbol{\parallel}v\boldsymbol{\parallel}_{\Im}<\infty\right\} \label{186}
\end{equation}

Note that the norm for the $\mathsf{k}^{\text{th}}$ order derivative has a weight $s^{\mathsf{k}}$ instead of $s^{\mathsf{k}+1}$ as expected, and that $v\in\Im$ if and only if $v\in C_{0}^{\mathsf{k}}[R,\,\infty)$
which decays at infinity at the following rate: 
\[
v=O\left(s^{-1}\right),\;\partial_{s}v=O\left(s^{-2}\right),\cdots,\;\partial_{s}^{\mathsf{k}-1}v=O\left(s^{-\mathsf{k}}\right),\, \partial_{s}^{\mathsf{k}}v=O\left(s^{-\mathsf{k}}\right)\quad\textrm{as}\; s\nearrow\infty
\]
For instance, $s^{-1}\in\Im$. Also, $\Im$ with the norm $\boldsymbol{\parallel}\cdot\boldsymbol{\parallel}_{\Im}$
is a Banach space. \\

In the following lemma, we estimate $\mathcal{Q}v$
in (\ref{167}) for $v\in\Im$, which is then used to bound the RHS of (\ref{169}).
\begin{lem}\label{l30}
Given $M>0$, there is $R=R\left(\mathsf{k},\,\mathcal{C},\,\mathcal{K},\, M\right)\geq1$
so that for any $v\in\Im$ with $\boldsymbol{\parallel}v\boldsymbol{\parallel}_{\Im}\leq M$,
we have 
\begin{equation}
\omega\left(v;\,\theta\right)\in\mathcal{K}\quad\forall\,\,\theta\in\left[0,\,1\right]\label{187}
\end{equation}
where $\omega\left(v;\,\theta\right)$
is defined by (\ref{168}) and $\mathcal{K}$ is a bounded neighborhood of $\left(\overrightarrow{1},\,0\right)$
in $\mathbb{R}^{n}$ stated in the main theorem. 
Moreover, $\mathcal{Q}v$ (defined in (\ref{167})) belongs to $C_{0}^{\mathsf{k}-2}[R,\,\infty)$ and satisfies
\[
\max\Big\{\parallel s^{5}\mathcal{Q}v\parallel_{L^{\infty}[R,\,\infty)},\;\parallel s^{6}\partial_{s}\mathcal{Q}v\parallel_{L^{\infty}[R,\,\infty)},\cdots,\;\parallel s^{\mathsf{k}+2}\partial_{s}^{\mathsf{k}-3}\mathcal{Q}v\parallel_{L^{\infty}[R,\,\infty)},
\]
\begin{equation}
\parallel s^{\mathsf{k}+2}\partial_{s}^{\mathsf{k}-2}\mathcal{Q}v\parallel_{L^{\infty}[R,\,\infty)}\Big\}\,\leq\, C\left(n,\,\mathsf{k},\mathcal{\, C},\,\parallel f\parallel_{C^{\mathsf{k}}\left(\mathcal{K}\right)},\, M\right)\label{188}
\end{equation}
\\
In addition, given a function $\tilde{v}\in\Im$ with $\boldsymbol{\parallel}\tilde{v}\boldsymbol{\parallel}_{\Im}\,\leq M$,
there holds

\noindent \resizebox{1.0\linewidth}{!}{
  \begin{minipage}{\linewidth}
  \begin{align}
\max\Big\{\parallel s^{5}\left(\mathcal{Q}v-\mathcal{Q}\tilde{v}\right)\parallel_{L^{\infty}[R,\,\infty)},\;\parallel s^{6}\left(\partial_{s}\mathcal{Q}v-\partial_{s}\mathcal{Q}\tilde{v}\right)\parallel_{L^{\infty}[R,\,\infty)},\cdots,\;\parallel s^{\mathsf{k}+2}\left(\partial_{s}^{\mathsf{k}-3}\mathcal{Q}v-\partial_{s}^{\mathsf{k}-3}\mathcal{Q}\tilde{v}\right)\parallel_{L^{\infty}[R,\,\infty)}, \label{189}
\end{align}
  \end{minipage}
}

\[
\parallel s^{\mathsf{k}+2}\left(\partial_{s}^{\mathsf{k}-2}\mathcal{Q}v-\partial_{s}^{\mathsf{k}-2}\mathcal{Q}\tilde{v}\right)\parallel_{L^{\infty}[R,\,\infty)}\Big\}
\]
\[
\leq\, C\left(n,\,\mathsf{k},\mathcal{\, C},\,\parallel f\parallel_{C^{\mathsf{k}+1}\left(\mathcal{K}\right)},\, M\right)\boldsymbol{\parallel}v-\tilde{v}\boldsymbol{\parallel}_{\Im}
\]
\end{lem}
\begin{proof}
In order to prove the decay, we would rescale the variables in such a way that the normalized variables are bounded by a constant depending only on
$n,\,\mathsf{k},\,\mathcal{C}$, $\left\Vert f\right\Vert_{C^{\mathsf{k}+1}\left(\mathcal{K}\right)}$ and $M$.

Let $R\geq1$ be a constant to be determined. For each $\mathit{\Lambda}\geq2R$,
consider the following change of variables:
\begin{equation}
s=\mathit{\Lambda}\xi\label{190}
\end{equation}
\begin{equation}
\widehat{v}(\xi)=\mathit{\Lambda}\, v\left(\mathit{\Lambda}\xi\right)\label{191}
\end{equation}
for $\xi\in\left[\frac{1}{2},\,1\right].$ 

From (\ref{185}), the condition $\boldsymbol{\parallel}v\boldsymbol{\parallel}_{\Im}\,\leq M$
implies that
\begin{equation}
\parallel\widehat{v}\parallel_{C^{\mathsf{k}-1}\left[\frac{1}{2},\,1\right]}\,\equiv\,\max\left\{ \parallel\widehat{v}\parallel_{L^{\infty}\left[\frac{1}{2},\,1\right]},\cdots,\;\parallel\partial_{\xi}^{\mathsf{k}-1}\widehat{v}\parallel_{L^{\infty}\left[\frac{1}{2},\,1\right]}\right\} \,\leq\, 2^k M\label{192}
\end{equation}
 
\[
\parallel\partial_{\xi}^{\mathsf{k}}\widehat{v}\parallel_{L^{\infty}\left[\frac{1}{2},\,1\right]}\,\,\leq\, 2^k M\mathit{\Lambda}
\]
Note that the bound for $\partial_{\xi}^{\mathsf{k}}\widehat{v}$ depends on the scale $\mathit{\Lambda}$ because the  $\mathsf{k}^{\text{th}}$ order derivative decays like ${s}^{-\mathsf{k}}$ instead of ${s}^{-\mathsf{k}-1}$ (see (\ref{185})).
And (\ref{168}) is translated into 
\[
\omega\left(v\left(s\right);\,\theta\right)=\left(\frac{1}{1+\mathit{\Lambda}^{-2}\theta\frac{\widehat{v}}{\sigma\xi}}\overrightarrow{1},\,\frac{-2\mathit{\Lambda}^{-2}\theta\sigma\xi\,\partial_{\xi}^{2}\widehat{v}}{1+\left(\sigma+\,\mathit{\Lambda}^{-2}\theta\,\partial_{\xi}\widehat{v}\right)^{2}}\right)
\]
which we denote by
\begin{equation}
\widehat{\omega}\left(\widehat{v}\left(\xi\right);\,\theta\right)\equiv\left(\frac{1}{1+\mathit{\Lambda}^{-2}\theta\frac{\widehat{v}}{\sigma\xi}}\overrightarrow{1},\,\frac{-2\mathit{\Lambda}^{-2}\theta\sigma\xi\,\partial_{\xi}^{2}\widehat{v}}{1+\left(\sigma+\,\mathit{\Lambda}^{-2}\theta\,\partial_{\xi}\widehat{v}\right)^{2}}\right)\label{193}
\end{equation}
Thus, by (\ref{193}) and (\ref{192}), (\ref{187}) holds if
$R$ is chosen sufficiently large (depending on $\mathsf{k},\,\mathcal{C},\,\mathcal{K},\, M$).

Next, we rescale (\ref{163}), (\ref{164}), (\ref{165}) and (\ref{167}) in the following way:
\begin{equation}
\widehat{\textbf{z}}\left(\widehat{v}\left(\xi\right),\,\theta\right)\equiv\mathit{\Lambda}^{-1}\textbf{z}\left(v\left(s\right),\,\theta\right)=\sigma\xi+\mathit{\Lambda}^{-2}\theta\,\widehat{v}\left(\xi\right)\label{194}
\end{equation}
\begin{equation}
\widehat{\textbf{p}}\left(\widehat{v}\left(\xi\right),\,\theta\right)\equiv\textbf{p}\left(v\left(s\right),\,\theta\right)=\sigma+\mathit{\Lambda}^{-2}\theta\,\partial_{\xi}\widehat{v}\label{195}
\end{equation}
\begin{equation}
\widehat{\textbf{q}}\left(\widehat{v}\left(\xi\right),\,\theta\right)\equiv\mathit{\Lambda}^{3}\textbf{q}\left(v\left(s\right),\,\theta\right)=\theta\,\partial_{\xi}^{2}\widehat{v}\label{196}
\end{equation}
\begin{equation}
\widehat{\mathcal{Q}v}(\xi)=\mathit{\Lambda}^{5}\,\mathcal{Q}v\left(\mathit{\Lambda}\xi\right)\label{197}
\end{equation}
Expand (\ref{197}) by the formula  (\ref{167}) and  write it in terms of (\ref{193}),
(\ref{194}), (\ref{195}) and (\ref{196}):
\begin{equation}
\widehat{\mathcal{Q}v}(\xi)=\left(\int_{0}^{1}\frac{\sigma\xi\,\partial_{nn}^{2}f\circ\widehat{\omega}}{\left(1+\widehat{\textbf{p}}^{2}\right)^{2}}\,\left(1-\theta\right)d\theta\right)\left(\partial_{\xi}^{2}\widehat{v}\right)^{2}\label{198}
\end{equation}
\[
+\left(\int_{0}^{1}\left(\frac{\sigma\xi\,\partial_{nn}^{2}f\circ\widehat{\omega}}{\mathit{\Lambda}^{4}}\frac{4\widehat{\textbf{q}}^{2}\widehat{\textbf{p}}^{2}}{\left(1+\widehat{\textbf{p}}^{2}\right)^{4}}+\,\frac{\partial_{n}f\circ\widehat{\omega}}{\mathit{\Lambda}^{2}}\frac{2\widehat{\textbf{q}}\left(1-3\widehat{\textbf{p}}^{2}\right)}{\left(1+\widehat{\textbf{p}}^{2}\right)^{3}}\right)\,\left(1-\theta\right)d\theta\right)\left(\partial_{\xi}\widehat{v}\right)^{2}
\]
\[
+\left(\int_{0}^{1}\frac{\left(\sum_{i,\, j=1}^{n-1}\sigma\xi\,\partial_{ij}^{2}f\circ\widehat{\omega}\right)+\,2\widehat{\textbf{z}}\,\sum_{i=1}^{n-1}\partial_{i}f\circ\widehat{\omega}}{\widehat{\textbf{z}}^{4}}\,\left(1-\theta\right)d\theta\right)\widehat{v}^{2}
\]
\[
+2\left(\int_{0}^{1}\left(\frac{\sigma\xi\,\partial_{nn}^{2}f\circ\widehat{\omega}}{\mathit{\Lambda}^{2}}\frac{-2\widehat{\textbf{q}}\widehat{\textbf{p}}}{\left(1+\widehat{\textbf{p}}^{2}\right)^{3}}+\,\partial_{n}f\circ\widehat{\omega}\frac{2\widehat{\textbf{p}}}{\left(1+\widehat{\textbf{p}}^{2}\right)^{2}}\right)\,\left(1-\theta\right)d\theta\right)\partial_{\xi}^{2}\widehat{v}\,\,\partial_{\xi}\widehat{v}
\]
\[
+2\left(\int_{0}^{1}\sum_{i=1}^{n-1}\frac{\sigma\xi\,\partial_{ni}^{2}f\circ\widehat{\omega}}{\left(1+\widehat{\textbf{p}}^{2}\right)\left(\widehat{\textbf{z}}\right)^{2}}\,\left(1-\theta\right)d\theta\right)\widehat{v}\,\,\partial_{\xi}^{2}\widehat{v}
\]
\[
-\,2\left(\int_{0}^{1}\sum_{i=1}^{n-1}\frac{\sigma\xi\,\partial_{ni}^{2}f\circ\widehat{\omega}}{\mathit{\Lambda}^{2}}\frac{2\widehat{\textbf{q}}\widehat{\textbf{p}}}{\left(1+\widehat{\textbf{p}}^{2}\right)^{2}\left(\widehat{\textbf{z}}\right)^{2}}\,\left(1-\theta\right)d\theta\right)\widehat{v}\,\,\partial_{\xi}\widehat{v}
\]
Then one could see that
\[
\widehat{\mathcal{Q}v}(\xi)\in C^{\mathsf{k}-2}\left[\frac{1}{2},\,1\right]
\]
By (\ref{192}), it satisfies

\begin{equation}
\parallel\widehat{\mathcal{Q}v}\parallel_{C^{\mathsf{k}-3}\left[\frac{1}{2},\,1\right]}\,\leq C\left(n,\,\mathsf{k},\mathcal{\, C},\,\parallel f\parallel_{C^{\mathsf{k}-1}\left(\mathcal{K}\right)},\, M\right)\label{199}
\end{equation}
\begin{equation}
\parallel\partial_{\xi}^{\mathsf{k}-2}\widehat{\mathcal{Q}v}\parallel_{L^{\infty}\left[\frac{1}{2},\,1\right]}\,\leq C\left(n,\,\mathsf{k},\mathcal{\, C},\,\parallel f\parallel_{C^{\mathsf{k}}\left(\mathcal{K}\right)},\, M\right)\mathit{\Lambda}\label{200}
\end{equation}
Note that in (\ref{200}), we have used the fact that $\mathsf{k}\geq3$
and $\partial_{\xi}^{\mathsf{k}-2}\widehat{\mathcal{Q}v}$
is linear in $\partial_{\xi}^{\mathsf{k}}\widehat{v}$.

Similarly, given a function $\tilde{v}\in\Im$ with $\boldsymbol{\parallel}\tilde{v}\boldsymbol{\parallel}_{\Im}\,\leq M$, we define $\widehat{\tilde{v}}(\xi)$ and $\widehat{\mathcal{Q}\tilde{v}}(\xi)$ in the same way as in (\ref{191}) and (\ref{197}). By doing subtraction of (\ref{198}), we get

\begin{equation}
\parallel\widehat{\mathcal{Q}v}-\widehat{\mathcal{Q}\tilde{v}}\parallel_{C^{\mathsf{k}-3}\left[\frac{1}{2},\,1\right]}\,\leq\, C\left(n,\,\mathsf{k},\mathcal{\, C},\,\parallel f\parallel_{C^{\mathsf{k}}\left(\mathcal{K}\right)},\, M\right)\parallel\widehat{v}-\widehat{\tilde{v}}\parallel_{C^{\mathsf{k}-1}\left[\frac{1}{2},\,1\right]}\label{201}
\end{equation}

\begin{equation}
\parallel\partial_{\xi}^{\mathsf{k}-2}\widehat{\mathcal{Q}v}-\partial_{\xi}^{\mathsf{k}-2}\widehat{\mathcal{Q}\tilde{v}}\parallel_{L^{\infty}\left[\frac{1}{2},\,1\right]}\label{202}
\end{equation}
\[
\leq\, C\left(n,\,\mathsf{k},\mathcal{\, C},\,\parallel f\parallel_{C^{\mathsf{k}+1}\left(\mathcal{K}\right)},\, M\right)\left(\parallel\partial_{\xi}^{\mathsf{k}}\widehat{v}-\partial_{\xi}^{\mathsf{k}}\widehat{\tilde{v}}\parallel_{L^{\infty}\left[\frac{1}{2},\,1\right]}+\mathit{\Lambda}\parallel\widehat{v}-\widehat{\tilde{v}}\parallel_{C^{\mathsf{k}-1}\left[\frac{1}{2},\,1\right]}\right)
\]

The conclusion follows immediately by undoing the change of variables for (\ref{199}), (\ref{200}),
(\ref{201}) and (\ref{202}).
\end{proof}

Now we are ready to show the existence of the problem (\ref{169}) and
(\ref{162}), which then yields the solution to (\ref{158}) and (\ref{160}) via (\ref{161}).

\begin{thm}\label{t31}
There exists $R=R\left(n,\,\mathsf{k},\mathcal{\, C},\,\mathcal{K},\,\parallel f\parallel_{C^{\mathsf{k}+1}\left(\mathcal{K}\right)}\right)\geq1$
and $u\in\Im$ such that 
\[
\mathcal{L}u=\frac{1+\sigma^{2}}{\partial_{n}f\left(\overrightarrow{1},\,0\right)}\left(\frac{f\left(\overrightarrow{1},\,0\right)}{\sigma s}-\,\sum_{i=1}^{n-1}\frac{\partial_{i}f\left(\overrightarrow{1},\,0\right)}{\sigma^{2}s^{2}}u+\,\mathcal{Q}u\right)\quad\textrm{on}\,\;[R,\,\infty)
\]
where $\Im$ is a subspace of $C_{0}^{\mathsf{k}}[R,\,\infty)$ defined
in (\ref{186}), $\mathcal{L}$ is the linear differential operator
defined in (\ref{170}) and $\mathcal{Q}$ is a nonlinear operator
defined in (\ref{167}).
Moreover, we have the following asymptotic formula: 
$$u\left(s\right)=\frac{f\left(\overrightarrow{1},\,0\right)}{\sigma s}+\overset{\circ}{\mathsf{u}}(s)$$
where the error term $\overset{\circ}{\mathsf{u}}(s)\in C_{0}^{\mathsf{k}}[R,\,\infty)$ satisfies

\noindent \resizebox{1.0\linewidth}{!}{
  \begin{minipage}{\linewidth}
  \begin{align*}
\parallel s^{3}\,\overset{\circ}{\mathsf{u}}(s)\parallel_{L^{\infty}[R,\,\infty)}+\parallel s^{4}\partial_{s}\overset{\circ}{\mathsf{u}}(s)\parallel_{L^{\infty}[R,\,\infty)}+\cdots+\parallel s^{\mathsf{k}+2}\partial_{s}^{\mathsf{k}-1}\overset{\circ}{\mathsf{u}}(s)\parallel_{L^{\infty}[R,\,\infty)}\,\leq C\left(n,\,\mathsf{k},\mathcal{\, C},\,\parallel f\parallel_{C^{\mathsf{k}}\left(\mathcal{\mathit{\mathcal{K}}}\right)}\right)
\end{align*}
 \end{minipage}
}

\[
\parallel s^{\mathsf{k}+1}\partial_{s}^{\mathsf{k}}\overset{\circ}{\mathsf{u}}(s)\parallel_{L^{\infty}[R,\,\infty)}\,\leq C\left(n,\,\mathsf{k},\mathcal{\, C},\,\mathcal{K},\,\parallel f\parallel_{C^{\mathsf{k}}\left(\mathcal{\mathit{\mathcal{K}}}\right)}\right)
\]
Note that $r\left(s\right)=\sigma s+u\left(s\right)$
solves
\[
f\left(\frac{1}{r}\overrightarrow{1},\;\frac{-\partial_{s}^{2}r}{1+\left(\partial_{\text{s}}r\right){}^{2}}\right)+\;\frac{1}{2}\left(s\,\partial_{s}r-r\right)=0\quad\textrm{on}\;[R,\,\infty)
\]
\[
\varrho\, r\left(\frac{s}{\varrho}\right)-\sigma s\overset{C_{\textrm{loc}}^{\mathsf{k}}}{\longrightarrow}0\quad\textrm{as}\;\varrho\searrow0
\]
\end{thm}
\begin{rem}\label{r32}
We find the asymptotic formula: 
$$r\left(s\right)\approx\sigma s+\frac{f\left(\overrightarrow{1},\,0\right)}{\sigma s}$$
by solving the following. Let
$r_{0}\left(s\right)=\sigma s$ and define $r_{1}\left(s\right)$
to be the solution to
\[
f\left(\frac{1}{r_{0}}\overrightarrow{1},\;\frac{-\partial_{s}^{2}r_{0}}{1+\left(\partial_{\text{s}}r_{0}\right){}^{2}}\right)+\;\frac{1}{2}\left(s\,\partial_{s}r_{1}-r_{1}\right)=0
\]
\[
\varrho\, r_{1}\left(\frac{s}{\varrho}\right)-\sigma s\overset{C_{\textrm{loc}}^{\mathsf{k}}}{\longrightarrow}0\quad\textrm{as}\;\varrho\searrow0
\]
Then $$r_{1}\left(s\right)=\sigma s+\frac{f\left(\overrightarrow{1},\,0\right)}{\sigma s}$$
It would be proved rigorously below that this formula really gives the asymptotic to our solution as $s\nearrow\infty$.
\end{rem}
\begin{proof}
Let $M>0$ and $R\geq1$ be constants to be chosen, and take $R$
sufficiently large (depending on $\mathsf{k},\,\mathcal{C},\,\mathcal{K},\, M$)
so that Lemma  \ref{l30} holds. Let $$\mathcal{\boldsymbol{\mathcal{B}}}=\left\{ v\in\Im\Big|\,\boldsymbol{\parallel}v\boldsymbol{\parallel}_{\Im}\leq M\right\} $$
be the closed ball of radius $M$ in the Banach space $\left(\Im,\,\boldsymbol{\parallel}\cdot\boldsymbol{\parallel_{\Im}}\right)$
defined in (\ref{185}) and (\ref{186}). 
By Lemma \ref{l29}, we can define a nonlinear map $\mathcal{F}$ so that for each $v\in\mathcal{\boldsymbol{\mathcal{B}}}$,
$\mathcal{F}v$ is the unique solution to the following Cauchy problem:
\begin{equation}
\mathcal{L}\left(\mathcal{F}v\right)=\frac{1+\sigma^{2}}{\partial_{n}f\left(\overrightarrow{1},\,0\right)}\left(\frac{f\left(\overrightarrow{1},\,0\right)}{\sigma s}-\,\sum_{i=1}^{n-1}\frac{\partial_{i}f\left(\overrightarrow{1},\,0\right)}{\sigma^{2}s^{2}}v+\,\mathcal{Q}v\right)\quad\textrm{on}\,\;[R,\,\infty)\label{203}
\end{equation}
\begin{equation*}
\frac{\mathcal{F}v(s)}{s}\rightarrow0,\;\left(s\,\partial_{s}\mathcal{F}v-\mathcal{F}v\right)\rightarrow0\quad\textrm{as}\; s\nearrow\infty\label{204}
\end{equation*}
Since $\mathcal{Q}v\in C_{0}^{\mathsf{k}-2}[R,\,\infty)$, $\mathcal{F}$
maps $\mathcal{\boldsymbol{\mathcal{B}}}$ into $C_{0}^{\mathsf{k}}[R,\,\infty)$ by Lemma \ref{l29}.
In fact, we would show that $\mathcal{F}v\in\mathcal{\boldsymbol{\mathcal{B}}}$
and $\mathcal{F}$ is a contraction on $\mathcal{\boldsymbol{\mathcal{B}}}$
if we choose $M$ and $R$ appropriately.\\

To start with, let
\begin{equation}
\overset{\circ}{\mathsf{v}}(s)=\mathcal{F}v(s)-\frac{f\left(\overrightarrow{1},\,0\right)}{\sigma s}\label{205}
\end{equation}
Then we have $\overset{\circ}{\mathsf{v}}\in C_{0}^{\mathsf{k}}[R,\,\infty)$ since both
$\mathcal{F}v$ and $s^{-1}$ belong to $C_{0}^{\mathsf{k}}[R,\,\infty)$. Also, by plugging
(\ref{205}) into (\ref{203}), we get 
\begin{equation}
\mathcal{L}\overset{\circ}{\mathsf{v}}=-2\frac{f\left(\overrightarrow{1},\,0\right)}{\sigma s^{3}}+\,\frac{1+\sigma^{2}}{\partial_{n}f\left(\overrightarrow{1},\,0\right)}\left(-\sum_{i=1}^{n-1}\frac{\partial_{i}f\left(\overrightarrow{1},\,0\right)}{\sigma^{2}}\frac{v}{s^{2}}+\,\mathcal{Q}v\right)\label{206}
\end{equation}
\begin{equation*}
\frac{\overset{\circ}{\mathsf{v}}}{s}\rightarrow0,\;\left(s\,\partial_{s}\overset{\circ}{\mathsf{v}}-\overset{\circ}{\mathsf{v}}\right)\rightarrow0,\quad\textrm{as}\; s\nearrow\infty\label{207}
\end{equation*}
Note that the RHS of (\ref{206}) has a stronger decay than that of (\ref{203}), which is an indication that it's better to study the equation of $\overset{\circ}{\mathsf{v}}$ instead of the equation of $\mathcal{F}v$.\\
Then apply Lemma  \ref{l29} (with $\gamma=3$) and Lemma  \ref{l30} to (\ref{206}) to get
\begin{equation}
\max\left\{ \parallel s^{3}\overset{\circ}{\mathsf{v}}\parallel_{L^{\infty}[R,\,\infty)},\,\parallel s^{4}\partial_{s}\overset{\circ}{\mathsf{v}}\parallel_{L^{\infty}[R,\,\infty)}\right\} \label{208}
\end{equation}
\[
\leq\,8\frac{\partial_{n}f\left(\overrightarrow{1},\,0\right)}{1+\sigma^{2}}\frac{f\left(\overrightarrow{1},\,0\right)}{\sigma}\,+\,4\sum_{i=1}^{n-1}\frac{\partial_{i}f\left(\overrightarrow{1},\,0\right)}{\sigma^{2}}\parallel s\, v\parallel_{L^{\infty}[R,\,\infty)}\,+\,4\frac{\parallel s^{5}\mathcal{Q}v\parallel_{L^{\infty}[R,\,\infty)}}{R^{2}}
\]
\[
\leq C\left(n,\,\mathsf{k},\mathcal{\, C},\,\parallel f\parallel_{C^{\mathsf{k}}\left(\mathcal{K}\right)},\, M\right)
\]
\begin{equation}
\parallel s^{3}\partial_{s}^{2}\overset{\circ}{\mathsf{v}}\parallel_{L^{\infty}[R,\,\infty)}\,\leq C\left(n,\,\mathsf{k},\mathcal{\, C},\,\parallel f\parallel_{C^{\mathsf{k}}\left(\mathcal{K}\right)},\, M\right)\label{209}
\end{equation}
Similarly, given a function $\tilde{v}\in\boldsymbol{\mathcal{B}}$, let
$\overset{\circ}{\mathsf{\tilde{v}}}=\mathcal{F}\tilde{v}-\frac{f\left(\overrightarrow{1},\,0\right)}{\sigma s}$ as in (\ref{205}).
Then we still have 
\begin{equation}
\mathcal{L}\overset{\circ}{\mathsf{\tilde{v}}}=-2\frac{f\left(\overrightarrow{1},\,0\right)}{\sigma s^{3}}+\,\frac{1+\sigma^{2}}{\partial_{n}f\left(\overrightarrow{1},\,0\right)}\left(-\sum_{i=1}^{n-1}\frac{\partial_{i}f\left(\overrightarrow{1},\,0\right)}{\sigma^{2}}\frac{\tilde{v}}{s^{2}}+\,\mathcal{Q}\tilde{v}\right)\label{210}
\end{equation}
\begin{equation*}
\frac{\overset{\circ}{\mathsf{\tilde{v}}}}{s}\rightarrow0,\;\left(s\,\partial_{s}\overset{\circ}{\mathsf{\tilde{v}}}-\overset{\circ}{\mathsf{\tilde{v}}}\right)\rightarrow0\quad\textrm{as}\; s\nearrow\infty\label{211}
\end{equation*}
Subtracting (\ref{210}) from (\ref{206}) to get
\begin{equation}
\mathcal{L}\left(\overset{\circ}{\mathsf{v}}-\overset{\circ}{\mathsf{\tilde{v}}}\right)=\frac{1+\sigma^{2}}{\partial_{n}f\left(\overrightarrow{1},\,0\right)}\left(-\sum_{i=1}^{n-1}\frac{\partial_{i}f\left(\overrightarrow{1},\,0\right)}{\sigma^{2}}\frac{v-\tilde{v}}{s^{2}}+\,\left(\mathcal{Q}v-\mathcal{Q}\tilde{v}\right)\right)\label{212}
\end{equation}
\begin{equation*}
\frac{\overset{\circ}{\mathsf{v}}-\overset{\circ}{\mathsf{\tilde{v}}}}{s}\rightarrow0,\;\left\{ \, s\,\partial_{s}\left(\mathsf{v-\overset{\circ}{\mathsf{\tilde{v}}}}\right)-\left(\mathsf{v-\overset{\circ}{\mathsf{\tilde{v}}}}\right)\,\right\} \rightarrow0\quad\textrm{as}\; s\nearrow\infty\label{213}
\end{equation*}
which, by Lemma  \ref{l29} (with $\gamma=3$) and Lemma  \ref{l30}, yields that
\begin{equation}
\max\left\{ \parallel s^{3}\left(\overset{\circ}{\mathsf{v}}-\overset{\circ}{\mathsf{\tilde{v}}}\right)\parallel_{L^{\infty}[R,\,\infty)},\;\parallel s^{4}\left(\partial_{s}\overset{\circ}{\mathsf{v}}-\partial_{s}\overset{\circ}{\mathsf{\tilde{v}}}\right)\parallel_{L^{\infty}[R,\,\infty)}\right\} \label{214}
\end{equation}
\[
\leq\,4\sum_{i=1}^{n-1}\frac{\partial_{i}f\left(\overrightarrow{1},\,0\right)}{\sigma^{2}}\parallel s\left(v-\tilde{v}\right)\parallel_{L^{\infty}[R,\,\infty)}+\,4\frac{\parallel s^{5}\left(\mathcal{Q}v-\mathcal{Q}\tilde{v}\right)\parallel_{L^{\infty}[R,\,\infty)}}{R^{2}}
\]
\[
\leq\, C\left(n,\,\mathsf{k},\mathcal{\, C},\,\parallel f\parallel_{C^{\mathsf{k}+1}\left(\mathcal{K}\right)},\, M\right)\boldsymbol{\parallel}v-\tilde{v}\boldsymbol{\parallel}_{\Im}
\]
\begin{equation}
\parallel s^{3}\left(\partial_{s}^{2}\overset{\circ}{\mathsf{v}}-\partial_{s}^{2}\overset{\circ}{\mathsf{\tilde{v}}}\right)\parallel_{L^{\infty}[R,\,\infty)}\,\leq\, C\left(n,\,\mathsf{k},\mathcal{\, C},\,\parallel f\parallel_{C^{\mathsf{k}+1}\left(\mathcal{K}\right)},\, M\right)\boldsymbol{\parallel}v-\tilde{v}\boldsymbol{\parallel}_{\Im}\label{215}
\end{equation}\\\\

Next, differentiate (\ref{206}), (\ref{212}) and use the commutator:
\begin{equation}
\partial_{s}\,\mathcal{L}-\mathcal{L}\,\partial_{s}=-\frac{1}{2}\frac{1+\sigma^{2}}{\partial_{n}f\left(\overrightarrow{1},\,0\right)}\partial_{s}\label{216}
\end{equation}
and also (\ref{208}), (\ref{209}), (\ref{214}), (\ref{215}) to get

\begin{equation}
\mathcal{L}\left(\partial_{s}\overset{\circ}{\mathsf{v}}\right)=\,6\frac{f\left(\overrightarrow{1},\,0\right)}{\sigma s^{4}}+\,\frac{1+\sigma^{2}}{\partial_{n}f\left(\overrightarrow{1},\,0\right)}\left(\frac{1}{2}\partial_{s}\overset{\circ}{\mathsf{v}}-\,\sum_{i=1}^{n-1}\frac{\partial_{i}f\left(\overrightarrow{1},\,0\right)}{\sigma^{2}}\,\partial_{s}\left(\frac{v}{s^{2}}\right)+\,\partial_{s}\mathcal{Q}v\right)\label{217}
\end{equation}
\begin{equation*}
\frac{\partial_{s}\overset{\circ}{\mathsf{v}}}{s}\rightarrow0,\;\left(s\partial_{s}^{2}\overset{\circ}{\mathsf{v}}-\partial_{s}\overset{\circ}{\mathsf{v}}\right)\rightarrow0\quad\textrm{as}\; s\nearrow\infty\label{218}
\end{equation*}

\noindent \resizebox{1.0\linewidth}{!}{
  \begin{minipage}{\linewidth}
  \begin{align}
\mathcal{L}\left(\partial_{s}\overset{\circ}{\mathsf{v}}-\partial_{s}\overset{\circ}{\mathsf{\tilde{v}}}\right)=\frac{1+\sigma^{2}}{\partial_{n}f\left(\overrightarrow{1},\,0\right)}\left(\frac{1}{2}\left(\partial_{s}\overset{\circ}{\mathsf{v}}-\partial_{s}\overset{\circ}{\mathsf{\tilde{v}}}\right)-\,\sum_{i=1}^{n-1}\frac{\partial_{i}f\left(\overrightarrow{1},\,0\right)}{\sigma^{2}}\,\partial_{s}\left(\frac{v-\tilde{v}}{s^{2}}\right)+\,\partial_{s}\left(\mathcal{Q}v-\mathcal{Q}\tilde{v}\right)\right)\label{219}
\end{align}
 \end{minipage}
}

\begin{equation*}
\frac{\partial_{s}\overset{\circ}{\mathsf{v}}-\partial_{s}\overset{\circ}{\mathsf{\tilde{v}}}}{s}\rightarrow0,\;\left\{ \, s\left(\partial_{s}^{2}\overset{\circ}{\mathsf{v}}-\partial_{s}^{2}\overset{\circ}{\mathsf{\tilde{v}}}\right)-\left(\partial_{s}\overset{\circ}{\mathsf{v}}-\partial_{s}\overset{\circ}{\mathsf{\tilde{v}}}\right)\,\right\} \rightarrow0\quad\textrm{as}\; s\nearrow\infty\label{220}
\end{equation*}
Now apply the gradient estimates in Lemma  \ref{l29} (with $\gamma=4$), Lemma  \ref{l30}  
(and also (\ref{208}), (\ref{214})) to (\ref{217}), (\ref{219}) to get

\noindent \resizebox{1.0\linewidth}{!}{
  \begin{minipage}{\linewidth}
  \begin{align}
\parallel s^{5}\partial_{s}^{2}\overset{\circ}{\mathsf{v}}\parallel_{L^{\infty}[R,\,\infty)}\leq\,\left\{\,24\frac{\partial_{n}f\left(\overrightarrow{1},\,0\right)}{1+\sigma^{2}}\frac{f\left(\overrightarrow{1},\,0\right)}{\sigma}\,+\,2\parallel s^{4}\partial_{s}\overset{\circ}{\mathsf{v}}\parallel_{L^{\infty}[R,\,\infty)}+\,4\sum_{i=1}^{n-1}\frac{\partial_{i}f\left(\overrightarrow{1},\,0\right)}{\sigma^{2}}\parallel s^{2}\partial_{s}v\parallel_{L^{\infty}[R,\,\infty)}\right.\label{221}
\end{align}
  \end{minipage}
}

\[
\left.+\,8\,\sum_{i=1}^{n-1}\frac{\partial_{i}f\left(\overrightarrow{1},\,0\right)}{\sigma^{2}}\parallel s\, v\parallel_{L^{\infty}[R,\,\infty)}+\,4\frac{\parallel s^{6}\partial_{s}\mathcal{Q}v\parallel_{L^{\infty}[R,\,\infty)}}{R^{2}}\,\right\}
\]
\[
\leq C\left(n,\,\mathsf{k},\mathcal{\, C},\,\parallel f\parallel_{C^{\mathsf{k}}\left(\mathcal{K}\right)},\, M\right)
\]
\begin{equation}
\parallel s^{4}\partial_{s}^{3}\overset{\circ}{\mathsf{v}}\parallel_{L^{\infty}[R,\,\infty)}\,\leq\, C\left(n,\,\mathsf{k},\mathcal{\, C},\,\parallel f\parallel_{C^{\mathsf{k}}\left(\mathcal{K}\right)},\, M\right)\label{222}
\end{equation}

\noindent \resizebox{1.0\linewidth}{!}{
  \begin{minipage}{\linewidth}
  \begin{align}
\parallel s^{5}\left(\partial_{s}^{2}\overset{\circ}{\mathsf{v}}-\partial_{s}^{2}\overset{\circ}{\mathsf{\tilde{v}}}\right)\parallel_{L^{\infty}[R,\,\infty)}\,\leq\,\left\{\,2\parallel s^{4}\left(\partial_{s}\overset{\circ}{\mathsf{v}}-\partial_{s}\overset{\circ}{\mathsf{\tilde{v}}}\right)\parallel_{L^{\infty}[R,\,\infty)}+\,4\sum_{i=1}^{n-1}\frac{\partial_{i}f\left(\overrightarrow{1},\,0\right)}{\sigma^{2}}\parallel s^{2}\left(\partial_{s}v-\partial_{s}\tilde{v}\right)\parallel_{L^{\infty}[R,\,\infty)}\right.\label{223}
\end{align}
 \end{minipage}
}

\[
\left.+8\,\sum_{i=1}^{n-1}\frac{\partial_{i}f\left(\overrightarrow{1},\,0\right)}{\sigma^{2}}\parallel s\left(v-\tilde{v}\right)\parallel_{L^{\infty}[R,\,\infty)}+\,4\frac{\parallel s^{6}\left(\partial_{s}\mathcal{Q}v-\partial_{s}\mathcal{Q}\tilde{v}\right)\parallel_{L^{\infty}[R,\,\infty)}}{R^{2}}\,\right\}
\]
\[
\leq\, C\left(n,\,\mathsf{k},\mathcal{\, C},\,\parallel f\parallel_{C^{\mathsf{k}+1}\left(\mathcal{K}\right)},\, M\right)\boldsymbol{\parallel}v-\tilde{v}\boldsymbol{\parallel}_{\Im}
\]
\begin{equation}
\parallel s^{4}\left(\partial_{s}^{3}\overset{\circ}{\mathsf{v}}-\partial_{s}^{3}\overset{\circ}{\mathsf{\tilde{v}}}\right)\parallel_{L^{\infty}[R,\,\infty)}\leq\,\, C\left(n,\,\mathsf{k},\mathcal{\, C},\,\parallel f\parallel_{C^{\mathsf{k}+1}\left(\mathcal{K}\right)},\, M\right)\boldsymbol{\parallel}v-\tilde{v}\boldsymbol{\parallel}_{\Im}\label{224}
\end{equation}
Note that by taking derivatives of the equation, (\ref{221}) and (\ref{223}) improve (\ref{209}) and (\ref{215}).\\\\

Continue the process (of differentiating equations and applying Lemma \ref{l29} and Lemma \ref{l30}) until we arrive at 

\noindent \resizebox{1.0\linewidth}{!}{
  \begin{minipage}{\linewidth}
  \begin{align}
\mathcal{L}\left(\partial_{s}^{\mathsf{k}-2}\overset{\circ}{\mathsf{v}}\right)=\,\partial_{s}^{\mathsf{k}-2}\left(-2\frac{f\left(\overrightarrow{1},\,0\right)}{\sigma s^{3}}\right)+\,\frac{1+\sigma^{2}}{\partial_{n}f\left(\overrightarrow{1},\,0\right)}\left(\frac{\mathsf{k}-2}{2}\partial_{s}^{\mathsf{k}-2}\overset{\circ}{\mathsf{v}}-\,\sum_{i=1}^{n-1}\frac{\partial_{i}f\left(\overrightarrow{1},\,0\right)}{\sigma^{2}}\partial_{s}^{\mathsf{k}-2}\left(\frac{v}{s^{2}}\right)+\,\partial_{s}^{\mathsf{k}-2}\mathcal{Q}v\right)\label{225}
\end{align}
 \end{minipage}
}

\begin{equation*}
\frac{\partial_{s}^{\mathsf{k}-2}\overset{\circ}{\mathsf{v}}}{s}\rightarrow0,\;\left(s\partial_{s}^{\mathsf{k}-1}\overset{\circ}{\mathsf{v}}-\partial_{s}^{\mathsf{k}-2}\overset{\circ}{\mathsf{v}}\right)\rightarrow0\quad\textrm{as}\; s\nearrow\infty\label{226}
\end{equation*}

\noindent \resizebox{1.0\linewidth}{!}{
  \begin{minipage}{\linewidth}
  \begin{align}
\mathcal{L}\left(\partial_{s}^{\mathsf{k}-2}\overset{\circ}{\mathsf{v}}-\partial_{s}^{\mathsf{k}-2}\overset{\circ}{\mathsf{\tilde{v}}}\right)=\left\{\,\frac{1+\sigma^{2}}{\partial_{n}f\left(\overrightarrow{1},\,0\right)}\left(\frac{\mathsf{k}-2}{2}\left(\partial_{s}^{\mathsf{k}-2}\overset{\circ}{\mathsf{v}}-\partial_{s}^{\mathsf{k}-2}\overset{\circ}{\mathsf{\tilde{v}}}\right)-\,\sum_{i=1}^{n-1}\frac{\partial_{i}f\left(\overrightarrow{1},\,0\right)}{\sigma^{2}}\partial_{s}^{\mathsf{k}-2}\left(\frac{v-\tilde{v}}{s^{2}}\right)\right)\right.\label{227}
\end{align}
 \end{minipage}
}

\[
\left.+\frac{1+\sigma^{2}}{\partial_{n}f\left(\overrightarrow{1},\,0\right)}\partial_{s}^{\mathsf{k}-2}\left(\mathcal{Q}v-\mathcal{Q}\tilde{v}\right)\right\}
\]
\begin{equation*}
\frac{\partial_{s}^{\mathsf{k}-2}\overset{\circ}{\mathsf{v}}-\partial_{s}^{\mathsf{k}-2}\overset{\circ}{\mathsf{\tilde{v}}}}{s}\rightarrow0,\;\left\{ \, s\left(\partial_{s}^{\mathsf{k}-1}\overset{\circ}{\mathsf{v}}-\partial_{s}^{\mathsf{k}-1}\overset{\circ}{\mathsf{\tilde{v}}}\right)-\left(\partial_{s}^{\mathsf{k}-2}\overset{\circ}{\mathsf{v}}-\partial_{s}^{\mathsf{k}-2}\overset{\circ}{\mathsf{\tilde{v}}}\right)\,\right\} \rightarrow0\quad\textrm{as}\; s\nearrow\infty\label{228}
\end{equation*}\\
By induction, we have for $j=0,\cdots,\mathsf{k}-2$,
\begin{equation}
\parallel s^{j+3}\partial_{s}^{j}\overset{\circ}{\mathsf{v}}\parallel_{L^{\infty}[R,\,\infty)}\leq\, C\left(n,\,\mathsf{k},\mathcal{\, C},\,\parallel f\parallel_{C^{\mathsf{k}}\left(\mathcal{K}\right)},\, M\right)\label{229}
\end{equation}
\begin{equation}
\parallel s^{j+3}\left(\partial_{s}^{j}\overset{\circ}{\mathsf{v}}-\partial_{s}^{j}\overset{\circ}{\mathsf{\tilde{v}}}\right)\parallel_{L^{\infty}[R,\,\infty)}\leq\, C\left(n,\,\mathsf{k},\mathcal{\, C},\,\parallel f\parallel_{C^{\mathsf{k}+1}\left(\mathcal{K}\right)},\, M\right)\boldsymbol{\parallel}v-\tilde{v}\boldsymbol{\parallel}_{\Im}\label{230}
\end{equation}
Again, apply Lemma  \ref{l29} (with $\gamma=\mathsf{k}+1$), Lemma
 \ref{l30} (and also (\ref{229}), (\ref{230})) to (\ref{225}) and (\ref{227}) to get
\begin{equation}
\parallel s^{\mathsf{k}+2}\partial_{s}^{\mathsf{k}-1}\overset{\circ}{\mathsf{v}}\parallel_{L^{\infty}[R,\,\infty)}\,\leq\,\left\{\,8\frac{\partial_{n}f\left(\overrightarrow{1},\,0\right)}{1+\sigma^{2}}\parallel s^{\mathsf{k}+1}\partial_{s}^{\mathsf{k}-2}\left(\frac{f\left(\overrightarrow{1},\,0\right)}{\sigma s^{3}}\right)\parallel_{L^{\infty}[R,\,\infty)}\right.\label{231}
\end{equation}
\[
+\,2\left(\mathsf{k}-2\right)\parallel s^{\mathsf{k}+1}\partial_{s}^{\mathsf{k}-2}\overset{\circ}{\mathsf{v}}\parallel_{L^{\infty}[R,\,\infty)}+\,4\sum_{i=1}^{n-1}\frac{\partial_{i}f\left(\overrightarrow{1},\,0\right)}{\sigma^{2}}\parallel s^{\mathsf{k}+1}\partial_{s}^{\mathsf{k}-2}\left(\frac{v}{s^{2}}\right)\parallel_{L^{\infty}[R,\,\infty)}
\]
\[
\left.+\,4\frac{\parallel s^{\mathsf{k}+2}\partial_{s}^{\mathsf{k}-2}\mathcal{Q}v\parallel_{L^{\infty}[R,\,\infty)}}{R}\,\right\}
\]
\[
\leq C\left(n,\,\mathsf{k},\mathcal{\, C},\,\parallel f\parallel_{C^{\mathsf{k}}\left(\mathcal{K}\right)},\, M\right)
\]

\begin{equation}
\parallel s^{\mathsf{k}+1}\partial_{s}^{\mathsf{k}}\overset{\circ}{\mathsf{v}}\parallel_{L^{\infty}[R,\,\infty)}\,\leq\, C\left(n,\,\mathsf{k},\mathcal{\, C},\,\parallel f\parallel_{C^{\mathsf{k}}\left(\mathcal{K}\right)},\, M\right)\label{232}
\end{equation}

\begin{equation}
\parallel s^{\mathsf{k}+2}\left(\partial_{s}^{\mathsf{k}-1}\overset{\circ}{\mathsf{v}}-\partial_{s}^{\mathsf{k}-1}\overset{\circ}{\mathsf{\tilde{v}}}\right)\parallel_{L^{\infty}[R,\,\infty)}\,\leq\,\Big\{\,2\left(\mathsf{k}-2\right)\parallel s^{\mathsf{k}+1}\left(\partial_{s}^{\mathsf{k}-2}\overset{\circ}{\mathsf{v}}-\partial_{s}^{\mathsf{k}-2}\overset{\circ}{\mathsf{\tilde{v}}}\right)\parallel_{L^{\infty}[R,\,\infty)}\label{233}
\end{equation}

\noindent \resizebox{1.0\linewidth}{!}{
  \begin{minipage}{\linewidth}
  \begin{align*}
+\,4\sum_{i=1}^{n-1}\frac{\partial_{i}f\left(\overrightarrow{1},\,0\right)}{\sigma^{2}}\parallel s^{\mathsf{k}+1}\partial_{s}^{\mathsf{k}-2}\left(\frac{v-\tilde{v}}{s^{2}}\right)\parallel_{L^{\infty}[R,\,\infty)}+\,4\frac{\parallel s^{\mathsf{k}+2}\left(\partial_{s}^{\mathsf{k}-2}\mathcal{Q}v-\partial_{s}^{\mathsf{k}-2}\mathcal{Q}\tilde{v}\right)\parallel_{L^{\infty}[R,\,\infty)}}{R}\,\Big\}
\end{align*}
 \end{minipage}
}

\[
\leq\, C\left(n,\,\mathsf{k},\mathcal{\, C},\,\parallel f\parallel_{C^{\mathsf{k}+1}\left(\mathcal{K}\right)},\, M\right)\boldsymbol{\parallel}v-\tilde{v}\boldsymbol{\parallel}_{\Im}
\]
\begin{equation}
\parallel s^{\mathsf{k}+1}\left(\partial_{s}^{\mathsf{k}}\overset{\circ}{\mathsf{v}}-\partial_{s}^{\mathsf{k}}\overset{\circ}{\mathsf{\tilde{v}}}\right)\parallel_{L^{\infty}[R,\,\infty)}\,\leq C\left(n,\,\mathsf{k},\mathcal{\, C},\,\parallel f\parallel_{C^{\mathsf{k}+1}\left(\mathcal{K}\right)},\, M\right)\boldsymbol{\parallel}v-\tilde{v}\boldsymbol{\parallel}_{\Im}\label{234}
\end{equation}
\\\\

From (\ref{229}), (\ref{231}) and
(\ref{232}), we get
\[
\mathcal{F}v\left(s\right)=\frac{f\left(\overrightarrow{1},\,0\right)}{\sigma s}+\overset{\circ}{\mathsf{v}}\left(s\right)\,\in\Im
\]
and
\begin{equation}
\boldsymbol{\parallel}\mathcal{F}v\boldsymbol{\parallel}_{\Im}\:\leq\:\boldsymbol{\parallel}\frac{f\left(\overrightarrow{1},\,0\right)}{\sigma s}\boldsymbol{\parallel}_{\Im}\,+\,\boldsymbol{\parallel}\overset{\circ}{\mathsf{v}}\boldsymbol{\parallel}_{\Im}\label{235}
\end{equation}

\noindent \resizebox{1.0\linewidth}{!}{
  \begin{minipage}{\linewidth}
  \begin{align*}
\leq\,\Big|\frac{f\left(\overrightarrow{1},\,0\right)}{\sigma}\Big|\boldsymbol{\parallel}s^{-1}\boldsymbol{\parallel}_{\Im}+\,\max\left\{ \,\frac{1}{R^{2}}\sum_{j=0}^{\mathsf{k}-1}\parallel s^{j+3}\partial_{s}^{j}\overset{\circ}{\mathsf{v}}\parallel_{L^{\infty}[R,\,\infty)},\;\frac{1}{R}\parallel s^{\mathsf{k}+1}\partial_{s}^{\mathsf{k}}\overset{\circ}{\mathsf{v}}\parallel_{L^{\infty}[R,\,\infty)}\,\right\} 
\end{align*}
 \end{minipage}
}

\[
\leq\,\Big|\frac{f\left(\overrightarrow{1},\,0\right)}{\sigma}\Big|\boldsymbol{\parallel}s^{-1}\boldsymbol{\parallel}_{\Im}\:+\:\frac{C\left(n,\,\mathsf{k},\mathcal{\, C},\,\parallel f\parallel_{C^{\mathsf{k}}\left(\mathcal{K}\right)},\, M\right)}{R}
\]
\\
Besides, from (\ref{230}), (\ref{233})
and (\ref{234}), we have 
\begin{equation}
\boldsymbol{\parallel}\mathcal{F}v-\mathcal{F}\tilde{v}\boldsymbol{\parallel}_{\Im}\:=\:\boldsymbol{\parallel}\overset{\circ}{\mathsf{v}}-\overset{\circ}{\mathsf{\tilde{v}}}\boldsymbol{\parallel}_{\Im}\,\label{236}
\end{equation}
\[
\leq\,\max\left\{ \,\frac{1}{R^{2}}\sum_{j=0}^{\mathsf{k}-1}\parallel s^{j+3}\left(\partial_{s}^{j}\overset{\circ}{\mathsf{v}}-\partial_{s}^{j}\overset{\circ}{\mathsf{\tilde{v}}}\right)\parallel_{L^{\infty}[R,\,\infty)},\;\frac{1}{R}\parallel s^{\mathsf{k}+1}\left(\partial_{s}^{\mathsf{k}}\overset{\circ}{\mathsf{v}}-\partial_{s}^{\mathsf{k}}\overset{\circ}{\mathsf{\tilde{v}}}\right)\parallel_{L^{\infty}[R,\,\infty)}\,\right\} 
\]
\[
\leq\frac{C\left(n,\,\mathsf{k},\mathcal{\, C},\,\parallel f\parallel_{C^{\mathsf{k}+1}\left(\mathcal{K}\right)},\, M\right)}{R}\boldsymbol{\parallel}v-\tilde{v}\boldsymbol{\parallel}_{\Im}
\]
Now we choose 
\[
M=\,\Big|\frac{f\left(\overrightarrow{1},\,0\right)}{\sigma}\Big|\boldsymbol{\parallel}s^{-1}\boldsymbol{\parallel}_{\Im}\,+\,\frac{1}{2}
\]
and take $R$ sufficiently large so that 
\[
\frac{C\left(n,\,\mathsf{k},\mathcal{\, C},\,\parallel f\parallel_{C^{\mathsf{k}+1}\left(\mathcal{K}\right)},\, M\right)}{R}\leq\frac{1}{2}
\]
Then we have $\mathcal{F}:\,\mathcal{\boldsymbol{\mathcal{B}}\rightarrow\mathcal{\boldsymbol{\mathcal{B}}}}$
is a contraction. 

By Banach fixed point theorem, there is a unique fixed point
$u$ of $\mathcal{F}$ in $\mathcal{\boldsymbol{\mathcal{B}}}$. Moreover, let
\[
\overset{\circ}{\mathsf{u}}(s)=\,\mathcal{F}u(s)\,-\,\frac{f\left(\overrightarrow{1},\,0\right)}{\sigma s}\,=u(s)\,-\,\frac{f\left(\overrightarrow{1},\,0\right)}{\sigma s}
\]
then by (\ref{229}), (\ref{231}) and
(\ref{232}), $\overset{\circ}{\mathsf{u}}\in C_{0}^{\mathsf{k}}[R,\,\infty)$ satisfies
\[
\Big\{\,\parallel s^{3}\,\overset{\circ}{\mathsf{u}}\parallel_{L^{\infty}[R,\,\infty)}+\parallel s^{4}\partial_{s}\overset{\circ}{\mathsf{u}}\parallel_{L^{\infty}[R,\,\infty)}+\cdots+\parallel s^{\mathsf{k}+2}\partial_{s}^{\mathsf{k}-1}\overset{\circ}{\mathsf{u}}\parallel_{L^{\infty}[R,\,\infty)}
\]
\[
+\parallel s^{\mathsf{k}+1}\partial_{s}^{\mathsf{k}}\overset{\circ}{\mathsf{u}}\parallel_{L^{\infty}[R,\,\infty)}\,\Big\}
\]
\[
\leq\, C\left(n,\,\mathsf{k},\mathcal{\, C},\,\parallel f\parallel_{C^{\mathsf{k}}\left(\mathcal{K}\right)}\right)
\]
\end{proof}
The following theorem is a direct result of Theorem \ref{t31}.
\begin{thm}\label{t33}
There exist $R=R\left(n,\,\mathsf{k},\mathcal{\, C},\,\mathcal{K},\,\parallel f\parallel_{C^{\mathsf{k}+1}\left(\mathcal{\mathit{\mathcal{K}}}\right)}\right)\geq1$
and $\overset{\circ}{\mathsf{u}}\in C_{0}^{\mathsf{k}}[R,\,\infty)$ such that 
\[
\Sigma\equiv\left\{\left. \left(\left(\sigma s\,+\,\frac{f\left(\overrightarrow{1},\,0\right)}{\sigma s}\,+\,\overset{\circ}{\mathsf{u}}\left(s\right)\right)\nu,\; s\right)\right|\,\nu\in\mathcal{\mathbf{S}}^{n-1},\, s\in[R,\,\infty)\right\} 
\]
is a rotationally symmetric $f$ self-shrinker which is $C^{\mathsf{k}}$asymptotic
to $\mathcal{C}$ at infinity. 
In addition, the corresponding self-similar
solution to the $f$ curvature flow is given by 
\[
\Sigma_{t}=\sqrt{-t}\,\Sigma=\left\{ \left. \left(\left(\sigma s\,-t\,\frac{f\left(\overrightarrow{1},\,0\right)}{\sigma s}\,+\,\overset{\circ}{\mathsf{u}}_{t}\left(s\right)\right)\nu,\; s\right)\right|\,\nu\in\mathcal{\mathbf{S}}^{n-1},\, s\in[\sqrt{-t}R,\,\infty)\right\} 
\]
for $t\in[-1,\,0)$, where $\overset{\circ}{\mathsf{u}}_{t}\left(s\right)\equiv\sqrt{-t}\,\overset{\circ}{\mathsf{u}}\left(\frac{s}{\sqrt{-t}}\right)$
satisfies 
\[
\parallel s^{3}\,\overset{\circ}{\mathsf{u}}_{t}\parallel_{L^{\infty}[\sqrt{-t}R,\,\infty)}+\parallel s^{4}\partial_{s}\overset{\circ}{\mathsf{u}}_{t}\parallel_{L^{\infty}[\sqrt{-t}R,\,\infty)}+\cdots+\parallel s^{k+2}\partial_{s}^{k-1}\overset{\circ}{\mathsf{u}}_{t}\parallel_{L^{\infty}[\sqrt{-t}R,\,\infty)}
\]
\[
\leq\, C\left(n,\,\mathsf{k},\mathcal{\, C},\,\parallel f\parallel_{C^{\mathsf{k}}\left(\mathcal{K}\right)}\right)\,\left(-t\right)^{2}
\]
\[
\parallel s^{k+1}\partial_{s}^{k}\overset{\circ}{\mathsf{u}}_{t}\parallel_{L^{\infty}[\sqrt{-t}R,\,\infty)}\,\leq\, C\left(n,\,\mathsf{k},\mathcal{\, C},\,\parallel f\parallel_{C^{\mathsf{k}}\left(\mathcal{K}\right)}\right)\,\left(-t\right)
\]
for all $t\in[-1,\,0)$.
\end{thm}

\vspace{0.3in}
\email{
\noindent Department of Mathematics, Rutgers University - Hill Center for the Mathematical Sciences 
110 Frelinghuysen Rd., Piscataway, NJ 08854-8019\\\\
E-mail address: \textsf{showhow@math.rutgers.edu}
}
\end{document}